\documentclass[ss]{imsart}
\setattribute{journal}{name}{}
\setattribute{journal}{issn}{}

\RequirePackage[OT1]{fontenc}
\RequirePackage{amsthm,amsmath}
\RequirePackage[numbers]{natbib}
\RequirePackage[colorlinks,citecolor=blue,urlcolor=blue]{hyperref}
\usepackage{amsfonts}
\usepackage{arydshln} 
\usepackage{multirow}
\usepackage{graphicx} 
\usepackage[usenames]{xcolor}

\numberwithin{equation}{section}
\theoremstyle{plain}

\newtheorem{prop}{Proposition}[section]
\newtheorem{cor}{Corollary}[section]
\newtheorem{lem}{Lemma}[section]
\theoremstyle{remark}

\endlocaldefs

\begin{document}

\begin{frontmatter}

\title{
Robust Principal Component Analysis in Hilbert spaces
}
\runtitle{
Robust PCA in Hilbert spaces
}

\begin{aug}
\author{\fnms{Ilaria} \snm{Giulini}\thanksref{t1}\ead[label=e1]{ilaria.giulini@me.com}}

\address{{\sc INRIA} Saclay\\
\printead{e1}}

\thankstext{t1}{The results presented in this paper 
were obtained while the author was preparing her PhD under the 
supervision of Olivier Catoni at the D\'epartement de Math\'ematiques et Applications, \'Ecole Normale Sup\'erieure, Paris, with the financial support of the 
R\'egion \^Ile de France.}

\runauthor{I. Giulini}

\affiliation{D\'epartement de Math\'ematiques et Applications, \'Ecole Normale Sup\'erieure, Paris, France}

\end{aug}

\begin{abstract}
{
We propose a stable version of Principal Component Analysis (PCA) in the general framework 
of a separable Hilbert space. 
It consists in interpreting the projection on the first eigenvectors 
as a step function applied to the spectrum of the covariance operator and in 
replacing it with a smooth cut-off of the eigenvalues. 
We study the problem from a statistical point of view, so that we assume that we
do not have direct access to the covariance operator but we have to estimate it from an 
i.i.d. sample. 
We provide some results on the quality of the approximation of our spectral cut-off 
in terms of the quality of the approximation of the eigenvalues of the covariance operator.
}
\end{abstract}


\tableofcontents
\end{frontmatter}

\section{Introduction}

Principal Component Analysis (PCA) is a classical tool for dimensionality reduction that relies on the spectral properties of the covariance matrix. 
The basic idea of PCA is to reduce the dimensionality of a dataset by projecting it into the space spanned by the directions of maximal variance, 
that are called its principal components.
Since this set of directions lies in the space spanned by the eigenvectors associated with the largest eigenvalues of the covariance matrix of the sample,
the dimensionality reduction is achieved by projecting the dataset along these eigenvectors, that in the following we call {\it first eigenvectors}.\\[1mm]
Several results can be found in the literature concerning 
the non-asymptotic setting.
These results rely on 
sharp non-asymptotic
bounds for the approximation error of the covariance matrix (e.g. Rudelson~\cite{Rud}, Tropp~\cite{JTr}, Vershynin~\cite{RVer}).\\[1mm]
We consider the setting of performing PCA on a separable Hilbert space.
This general framework includes the 
analysis of samples in a functional space (PCA for functional data, Ramsay and Silverman~\cite{amsay}) and of samples embedded in a reproducing kernel Hilbert space.
The latter is for example the case of kernel PCA, 
that uses the kernel trick to embed the dataset in a reproducing kernel Hilbert space 
in order to 
get a representation that simplifies the geometry of the dataset 
(e.g. Sch\"{o}lkopf, Smola, M\"{u}ller~\cite{Smul}, Zwald, Bousquet, Blanchard~\cite{ZBB}, Shawe-Taylor, Williams, Cristianini, Kandola~\cite{StWCK02}, ~\cite{StWCK05}).\\[1mm]
In our setting the goal is to perform PCA on the covariance operator
\begin{equation}\label{cov}
\Sigma u = \mathbb E \Bigl[ \big\langle u , X -\mathbb E[X] \big\rangle \ (X -\mathbb E[X])  \Bigr]
\end{equation}
where $\mathbb E$ is the expectation with respect to the law of the random vector $X$. 
The first assumption we make is that the law of $X$ is unknown, so that we can not work directly with $\Sigma$ 
but we have to construct an estimator from an i.i.d. sample drawn according to the unknown probability distribution of $X$.\\
Results concerning the estimation of the spectral projectors of the covariance operator
by their empirical counterpart in a Hilbert space can be found in Koltchinskii, Lounici  ~\cite{KoltLou2}, ~\cite{KoltLou3}.
The authors study the problem in the case of Gaussian centered random vectors, based on the bounds obtained in ~\cite{KoltLou1}, 
and in the setting where both the sample size $n$ and the trace of the covariance operator are large. \\[1mm]
One question that arises in standard PCA is how to determine the number of relevant components
and a method is based on the difference in magnitude between successive eigenvalues.
In particular, the bounds on the approximation error depends on the inverse of the size of the gap between the last eigenvalue we consider 
and the following one. 
Our goal is to propose a robust version of PCA.
Here the meaning of {\it robust} is not in the sense of Cand\`es, Li, Ma, Wright in ~\cite{CLMW} where they show that 
it is possible to recover the principal components of a data matrix in the case 
where the observations are contained in a low-dimensional space but arbitrarily corrupted by additive noise. 
In our context, the idea is to interpret the projector on the space spanned by the largest eigenvectors as a step function applied 
to the spectrum of the covariance operator and to replace it with a smooth cut-off of the eigenvalues of $\Sigma$ via a 
Lipschitz function.
This allows avoiding the assumption of a large eigengap and in particular in our results 
the size of the spectral gap is replaced by the inverse of
the Lipschitz constant of the function that defines the cut-off. 
The bounds we propose depend on the quality of the approximation of the eigenvalues of the covariance operator
and we use some recent results of~\cite{tesiGiulini} that appear in ~\cite{Giulini15}, to characterize this estimate. 
These results rely on a PAC-Bayesian approach to construct a robust estimator of the covariance operator (which is not the
classical empirical estimator obtained by replacing the unknown probability distribution with the sample distribution)
which provides non-asymptotic bounds for $\langle \Sigma u, u\rangle$, uniformly on $u$. 
Note that similar techniques are used in ~\cite{Catoni16} to construct a robust estimator of the Gram matrix in the finite-dimensional setting.
\\ 
A different robust estimator has been suggested by Minsker in ~\cite{SMin}. 
The construction of such an estimator is based on the geometric median and it is used to provide non-asymptotic bounds on the approximation error of spectral projectors.

\vskip1mm
\noindent 
The paper is organized as follows. 
In section ~\ref{bound_eigen} we provide a bound on the approximation error of the eigenvalues of the covariance operator 
based on some results in~\cite{Giulini15}. 
We then introduce in section~\ref{res} the robust version of PCA and we characterize the quality of the approximation 
in terms of the quality of the approximation of the eigenvalues.
Finally in appendix \ref{appx} we introduce some general results on orthogonal projectors.

\section{Estimate of the eigenvalues}\label{bound_eigen}

The main goal of this section is to provide a bound for the approximation error of the eigenvalues of the covariance operator. 
For the sake of simplicity we consider the finite-dimensional case and we provide a non-asymptotic bound that 
does not depend explicitly on the dimension of the ambient space, so that it can be generalized to any infinite-dimensional Hilbert space. \\[1mm]
From now on let $X \in \mathbb R^d$ be a random vector of (unknown) law $\mathrm P \in \mathcal M_+^1(\mathbb R^d)$ and assume it to be centered
{\footnote{The same kind of results hold in the case where $X$ is not centered and its expectation is unknown. 
For further details we refer to the end of Section ~\ref{sec_back} and to ~\cite{Giulini15} and ~\cite{tesiGiulini}.}.
The $d\times d$ covariance matrix is then given by
\[
\Sigma = \mathbb E \bigl[ XX ^{\top}  \bigr] \in M_d(\mathbb R).
\]
Observe that, if $\lambda_1 \geq \dots \geq \lambda_d$ denote the eigenvalues of $\Sigma$ and 
$p_1, \dots, p_d$ the corresponding orthonormal basis of eigenvectors, 
\[
\lambda_i = p_i^{\top} \Sigma p_i.
\]

\subsection{Background}\label{sec_back}

We introduce some results obtained in ~\cite{Giulini15} for the 
estimate of the energy $\theta^{\top} \Sigma \theta$, uniformly on the directions $\theta \in \mathbb R^d$,
using a robust estimator of $\Sigma$.
Such an estimator is constructed as the minimal quadratic form (in terms of the Frobenius norm{\footnote{Given $T$ a bounded linear operator, the Hilbert-Schmidt, or Frobenius, norm
of $T$ is defined as
$\|T\|_{HS}^2 = \mathbf{Tr}(T^{*}T)$ 
and $\| T\|_{\infty} \leq \| T\|_{HS}.$}}
) 
in the confidence region of $\theta^{\top} \Sigma \theta$ on a finite number of directions $\theta.$ 
To describe this confidence interval we have to introduce some notation. 
\\[1mm]
Let $X_1, \dots, X_n \in \mathbb R^d$ be an  i.i.d. sample drawn according to $\mathrm P$.
Given $\theta \in \mathbb R^d$, consider the classical empirical estimator of $\theta^{\top}\Sigma\theta$ 
truncated with the influence function
\[
\psi(t)= \begin{cases} \log(2) & \text{if } t\geq 1\\
-\log\left( 1-t+\frac{t^2}{2}\right) & \text{if } 0\leq t\leq 1\\
-\psi(-t) & \text{if } t\leq 0
 \end{cases}
\]
and centered with a parameter $\lambda>0$,
that is
\[
\frac{1}{n}\sum_{i=1}^n\psi\left( \langle \theta, X_i\rangle^2 -\lambda\right).
\]
Based on $\widehat \alpha_{\theta, \lambda}$, solution of 
\[
\sum_{i=1}^n\psi\left( \alpha^2 \langle \theta, X_i\rangle^2 -\lambda\right)=0,
\]
we define in ~\cite{Giulini15} upper and lower bounds, denoted respectively 
by $B_+(\theta)$ and $B_-(\theta)$,
such that, with probability at least $1-2\epsilon$, for any $\theta \in \mathbb R^d$,  
\begin{equation}\label{eq1_gm}
B_-(\theta) \leq \theta^\top \Sigma \theta \leq B_+(\theta).
\end{equation}
These bounds are obtained using an adaptive choice of $\lambda$ for each bound. 
In fact
\[
B_-(\theta) = \max_{\lambda \in \Lambda} B_-(\theta, \lambda) \quad \text{and} \quad 
B_+(\theta) = \min_{\lambda \in \Lambda} B_+(\theta, \lambda) 
\]
where $B_-(\theta, \lambda)$ and $B_+(\theta, \lambda)$ depend on $\lambda$ and $\Lambda$ is a suitable finite grid
of candidate values for $\lambda$. 
An oracle  for $\lambda$ (see the proofs in ~\cite{Giulini15}) is  
\[
\lambda_* = \sqrt{ \frac{2}{n (\kappa - 1)} \Bigg(\frac{0.73 \ \mathbb E [\| X\|^4]^{1/2} \|\theta\|^2}{\kappa^{1/2} \max\{ \theta^\top \Sigma \theta, \sigma\}} + 4.35 + \log(\epsilon^{-1}) \Bigg)},
\]
where $\sigma>0$ is a threshold and 
\[
\kappa  = \sup_{\substack{\theta \in \mathbb{R}^d \\ \mathbb E [ \langle \theta, X \rangle^2 ] > 0}} \frac{ \mathbb E \bigl[ \langle \theta , X \rangle^4 \bigr]}{
\mathbb E \bigl[ \langle \theta, X \rangle^2 \bigr]^2}.
\] 
This means that the accuracy of $B_-(\theta)$ and $B_+(\theta)$ is deduced from the fact that, by construction, 
\[
B_-(\theta, \lambda_*) \leq B_-(\theta) \leq B_+(\theta) \leq B_+(\theta, \lambda_*). 
\]
From a practical point of view, we can set the parameter $\lambda$ as
\[
\lambda = m\  \sqrt {\frac{1}{v} \left[ \frac{2}{n} \log(\epsilon^{-1}) \left(1-  \frac{2}{n} \log(\epsilon^{-1})  \right)  \right]}
\] 
where $\displaystyle m = \frac{1}{n}\sum_{i=1}^n \langle \theta, X_i\rangle^2$, 
$\displaystyle v = \frac{1}{n-1} 
\sum_{i=1}^n \left( \langle \theta, X_i\rangle^2 - m \right)^2$ and $ \epsilon = 0.1$.
However this choice is not fully mathematically justified but it is close to optimal 
for the estimation in a single direction $\theta$ from an empirical sample distribution, according to 
\cite{Catoni12}.\\[1mm]
Consider any symmetric matrix $Q$ such that 
\begin{equation}\label{eq_sol}
B_-(\theta) \leq \theta^\top Q \theta \leq B_+(\theta) \quad \theta \in \Theta_\delta
\end{equation}
where $\Theta_\delta$ is a finite $\delta$-net of the unit sphere $\mathbb S_d = \left\{ \theta \in \mathbb R^d \ | \ \| \theta\|=1\right\}$, meaning that
\[
\sup_{\theta\in \mathbb S_d} \min_{\xi \in\Theta_\delta }\| \theta- \xi\| \leq \delta.
\]
Decreasing $\delta$ 
increases the quality of the estimator 
but increases also its computation cost. 
We choose as an estimator the matrix $\widehat \Sigma$ that satisfies equation ~\eqref{eq_sol} and which is minimal in terms of the Frobenius norm. 
Since the constraints in equation ~\eqref{eq_sol} are satisfied at least by $\Sigma$, the solution exists and 
an explicit computation gives
\[
\widehat \Sigma = \sum_{\theta \in \Theta_\delta} \left[  \widehat \xi_+(\theta) -  \widehat \xi_-(\theta) \right] \theta \theta^{\top}
\]
where
\begin{multline*}
\left(\widehat \xi_+ , \widehat\xi_- \right) \\
= \arg\max_{ \xi_+ , \xi_- \in \left( \mathbb R_+^2\right)^{\Theta_\delta}} \Bigg\{ -\frac{1}{2} \sum_{(\theta, \theta') \in \Theta_\delta^2} \left[  \xi_+(\theta) - \xi_-(\theta)\right]  \left[  \xi_+(\theta') - \xi_-(\theta')\right]\langle \theta, \theta'\rangle^2 \\
+ \sum_{\theta \in \Theta_\delta}\left[\xi_+(\theta)B_-(\theta) - \xi_-(\theta)B_+(\theta) \right]\Bigg\}.
\end{multline*}
From a study of the deviations of $B_-(\theta)$ and $B_+(\theta)$, 
the following result is proved in ~\cite{Giulini15}.

\begin{prop} 
\label{prop_giu}
Let $\displaystyle \sigma \leq \frac{ 100 \, \kappa \mathrm{Tr}(\Sigma)}{n/128 - 4.35 -  \log \bigl( \epsilon^{-1} \bigr)}$
and $\kappa \geq 3/2$. 
Define
\[
\gamma(t) 
= \sqrt{ \frac{2.032 (\kappa-1)}{n} \Biggl( \frac{0.73 \ \mathrm{Tr}(\Sigma)}{t} + 4.35 + \log(\epsilon^{-1}) \Biggr)} + \sqrt{ \frac{98.5 \, \kappa \mathrm{Tr}(\Sigma)}{ nt}}
\]
and
\[
\eta(t) = 
\begin{cases} 
\displaystyle \frac{ \gamma(\max \{ t, \sigma \} )}{1 - 4 \, \gamma( \max \{ t, \sigma \} )} &  \bigl[ 6 + (\kappa-1)^{-1} \bigr] \gamma( \max \{ t, \sigma \} ) \leq 1 \\ 
+ \infty & \text{ otherwise.}
\end{cases} 
\]
With probability at least $1 - 2 \epsilon$, 
for any $\theta \in \mathbb{S}_d$, 
\begin{align*}
\Bigl\lvert \max \{ \theta^{\top} \widehat \Sigma \theta, \sigma \}  - 
\max \{ \theta^{\top} \Sigma \theta, \sigma \}  \Bigr\rvert 
& \leq 2 \max \bigl\{ \theta^{\top} \Sigma \theta, \sigma \bigr\} \ \eta \bigl( 
\theta^{\top} \Sigma \theta\bigr) + 5 \delta\lVert \Sigma \rVert_{HS}, \\
\Bigl\lvert \max \{ \theta^{\top} \widehat \Sigma \theta, \sigma \} - 
\max \{ \theta^{\top} \Sigma \theta, \sigma \}  \Bigr\rvert 
& \leq 2 \max \bigl\{ \theta^{\top}\widehat \Sigma \theta, \sigma \bigr\} \ \eta \bigl(\min\{
\theta^{\top} \widehat \Sigma \theta, s_4\} \bigr) + 5 \delta \lVert \Sigma \rVert_{HS},
\end{align*}
where $\| \cdot\|_{HS}$ denotes the Frobenius norm.
\end{prop}

\vskip2mm
\noindent
We conclude observing that in the non-centered case the covariance matrix has the form 
\[
\Sigma = \mathbb E \bigl[ \left(X - \mathbb E[X] \right)\left(X - \mathbb E[X] \right) ^{\top}  \bigr] 
\]
where the expectation $\mathbb E[X]$ is also unknown. 
However we do not have to estimate directly $\mathbb E[X]$ but we can find a workaround observing that 
\[
\theta^{\top} \Sigma \theta = \frac{1}{2} \mathbb E \left[\langle \theta , X-X'\rangle^2  \right]
\]
where $X'$ is an independent copy of $X.$
Thus, we can use the results proposed for the centered case using as a sample 
$\{ X_i - X_{i+n/2}\}_{i=1, \dots, n/2}.$
For further details we refer to
~\cite{Giulini15}.

\subsection{Main results}

We use the robust estimator $\widehat \Sigma$ introduced in the previous section 
to estimate the eigenvalues of the covariance matrix $\Sigma$,
recalling that each eigenvalue $\lambda_i$ of $\Sigma$ can be written as $p_i^{\top}\Sigma p_i$, where $p_i$ is the corresponding eigenvector. 
Moreover, since we do not know whether $\widehat \Sigma$ is non-negative, we use as an estimator of $\Sigma$ the positive part of 
$\widehat \Sigma$, denoted by $\widehat \Sigma_+$. 
Let $\widehat \lambda_{1} \geq \dots \geq \widehat \lambda_{d}$ be the eigenvalues of $\widehat \Sigma_+$.

\begin{prop}
\label{prop1.23eig} 
Let $\sigma$ be as in Proposition ~\ref{prop_giu} and $\kappa \geq 3/2$. 
With probability at least $1 - 2 \epsilon$, 
for any $i = 1, \dots, d$, the two following inequalities 
hold together
\begin{align*}
\bigl\lvert \max\{ \lambda_i , \sigma\}- \max\{ \widehat{\lambda}_i , \sigma\}\bigr\rvert & \leq 
2 \max\{ \lambda_i , \sigma\} \ \eta (\lambda_i) + 5 \delta \lVert \Sigma \rVert_{HS}, \\ 
\bigl\lvert\max\{ \lambda_i, \sigma\}-\max\{  \widehat{\lambda}_i, \sigma\} \bigr\rvert & \leq 
2 \max\{ \widehat{\lambda}_i , \sigma\} \ \eta \bigl(\min\{ \widehat{\lambda}_i , s_4\}\bigr) + 5 \delta \lVert \Sigma \rVert_{HS}. 
\end{align*}
Consequently, 
\begin{align*}
|\lambda_i - \widehat \lambda_i | & \leq 2 \max\{ \lambda_i , \sigma\}
\ \eta \bigl( \lambda_i \bigr) + 5 \delta \lVert \Sigma \rVert_{HS} + \sigma,\\
|\lambda_i - \widehat \lambda_i  | & \leq 2 \max\{\widehat \lambda_i , \sigma\}  
\ \eta \bigl(\min\{ \widehat{\lambda}_i , s_4\}\bigr)  + 5 \delta 
\lVert \Sigma \rVert_{HS} + \sigma.
\end{align*}
\end{prop}

\vskip2mm
\noindent
For the proof we refer to section ~\ref{pf_eig}.\\[1mm]
We conclude the section by making some comments on the above bound. 
According to the choice of $\sigma$ in Proposition ~\ref{prop_giu}, we have $\gamma(\sigma)\leq 1/8.$ 
Define
\[
B(t) =  2 \max\{ t, \sigma\} \ \eta\left( \min\{t, s_4\} \right) + 5 \delta \lVert \Sigma \rVert_{HS} + \sigma.
\]
We first observe that since the function 
$t \mapsto \max\{ t, \sigma\} \ \eta( \min\{t, s_4\})$ 
is non-decreasing for any $t \in \mathbb R_+$, then $B(\lambda_i) \leq B(\lambda_1)$ for any $i \geq 1$.  
Moreover since $\eta$ is non-increasing and $\eta(t) \leq 1/4$, for any $a \in \mathbb{R}_+$, we have 
$B(t + a) \leq B(t)  + a/2$. Moreover, the bound $B$ does not depend on the dimension $d$ of the ambient space. 
For this reason the result presented in Proposition~\ref{prop1.23eig} can be generalized to any infinite-dimensional Hilbert space with the only 
additional assumption
that the trace of $\Sigma$ is finite. 

\section{Robust PCA}\label{res}

Let us consider the general framework of a separable Hilbert space $\mathcal H$.
Let $\mathrm P \in \mathcal M_+^1(\mathcal H)$ be a probability distribution on $\mathcal H$ and 
let $X \in \mathcal H$ be a random vector of law $\mathrm P$.
Let $\Sigma$ denote the covariance operator defined in equation ~\eqref{cov} and consider any estimator $\widehat \Sigma$ of $\Sigma$
such that a bound of the type of Proposition~\ref{prop_giu} holds, meaning that,
with probability at least $1-\epsilon$, for any $u \in \mathcal S$, 
the unit sphere of $\mathcal{H}$,  
\begin{equation}\label{b_B}
\aligned
\left|\langle \Sigma u, u \rangle_{\mathcal H} - \langle\widehat \Sigma u, u \rangle_{\mathcal H} \right| 
& \leq B \Bigl( \langle \Sigma u, u \rangle_{\mathcal H} \Bigr)\\
 \left|\langle \Sigma u, u \rangle_{\mathcal H} - \langle\widehat \Sigma u, u \rangle_{\mathcal H} \right| 
&\leq B \Bigl( \langle \widehat \Sigma u, u \rangle_{\mathcal H} \Bigr)
\endaligned
\end{equation}
where $B$ is a non-decreasing function such that $B(t+a) \leq B(t) + a/2$.\\
Let $\lambda_1\geq\lambda_2 \geq \dots \geq 0$ denote the eigenvalues of $\Sigma$  
and 
$\widehat \lambda_1\geq \widehat \lambda_2 \geq \dots \geq 0$ those of $\widehat \Sigma_+$.
Following the same arguments as in the proof of Proposition \ref{prop1.23eig}, 
we deduce from inequalities \eqref{b_B} 
that with probability at least $1-\epsilon$, for any $i\geq 1,$
\begin{equation}\label{eq_eigen}
\aligned
|\lambda_i - \widehat \lambda_i |& \leq B(\lambda_i)\\
|\lambda_i - \widehat \lambda_i |& \leq B(\widehat \lambda_i).
\endaligned
\end{equation}
\vskip2mm
\noindent
As already said,
in the standard PCA setting, we assume that there is a gap in the spectrum of the covariance operator 
meaning that 
\[
\lambda_r - \lambda_{r+1} >0.
\]
Let $\Pi_r$ denote the orthogonal projector on the $r$ largest eigenvectors of $\Sigma$
and $\widehat \Pi_r$ the orthogonal projector on the $r$ largest eigenvectors of $\widehat \Sigma$.
The quality of the approximation 
depends on the inverse of the size of the eigengap. 
Next proposition shows that we can bound the approximation error in terms of
the number of relevant components $r$ and of the bound for the eigenvalues introduced in equation ~\eqref{eq_eigen}.

\begin{prop}\label{st_pca}  
For any positive integer $r$, with probability at least $1-\epsilon,$
\[
\|\Pi_r-\widehat \Pi_r\|_{\infty} \leq  \frac{\sqrt{2r}}{\lambda_{r}-\lambda_{r+1} }   B\left(\lambda_1 \right)
\]
where $\lambda_1$ is the largest eigenvalue of the covariance operator. 
\end{prop}

\vskip1mm
\noindent
For the proof we refer to section~\ref{pf1}.
Note that a similar result can also be obtained using a Davis-Kahan type result, e.g.
Blanchard, Zwald \cite{ZwaBla06a}.

\vskip2mm
\noindent
Observe that the proposition applies to the estimator and the bound
stated in Proposition~\ref{prop1.23eig}, providing a robust estimate 
for the PCA projectors.

\vskip2mm
\noindent
Remark also that the above bound relates the quality of the robust 
estimation of the orthogonal projector $\Pi_r$ 
to the size of the spectral gap. 
In particular, the larger the eigengap, the better the approximation is. 
The size of the eigengap can be estimated from the eigenvalues of $\widehat \Sigma$, 
which are close to those of $\Sigma$ according to equation~\eqref{eq_eigen}. 
\\[1mm]
In order to avoid the requirement of a large spectral gap, we  
propose to view the projector $\Pi_r$ as a step function applied 
to the spectrum of the covariance operator and to replace 
$\Pi_r$ with a smooth cut-off of the eigenvalues of $\Sigma$ via a 
Lipschitz function.
More specifically, we have in mind to apply to 
the spectrum of $\Sigma$ a Lipschitz function that takes the value one 
on the largest eigenvalues and the value zero on the smallest ones.\\[1mm]
Recalling that the covariance operator writes as
\[
\Sigma u= \sum_{i=1}^{+\infty} \lambda_i \langle u, p_i \rangle_{\mathcal H} \ p_i,
\]
in the following we write $f(\Sigma)$ to denote the operator 
\[
f(\Sigma) u= \sum_{i=1}^{+\infty} f( \lambda_i ) \ \langle u, p_i \rangle_{\mathcal H} \ p_i.
\]
The following proposition holds.
\begin{prop}\label{supnorm}{\bf (Operator norm)}
With probability at least $1-\epsilon$, 
for any $1/L$-Lipschitz function $f$, 
\begin{align*}
\| f(\Sigma)- f(\widehat \Sigma) \|_{\infty}& \leq \min_{\tau \geq 1} L^{-1}\left( B\left(\lambda_1 \right)+ \sqrt{4\tau B\left(\lambda_1 \right)^2 + 2\sum_{i=\tau+1}^{+\infty} \lambda_i^2} \; \right),
\end{align*}
 where $B$ is defined in equation ~\eqref{eq_eigen}. 
\end{prop}

\vskip2mm
\noindent
For the proof we refer to section~\ref{pf3}.\\[1mm]
Observe that with respect to the bound obtained in Proposition~\ref{st_pca}, we have replaced the size of the gap
with the inverse of the Lipschitz constant. 
Moreover, in the case when there exists a gap $\lambda_r- \lambda_{r+1}>0$, there exists a Lipschitz function $f$ 
such that $\Pi_r = f(\Sigma)$ and whose Lipschitz constant is exactly the inverse of the size of the gap. 
Otherwise, if we want to use $f$ with a better Lipschitz constant, we have to approximate $\Pi_r$ with the smoother
approximate projection $f(\Sigma)$. 

\vskip2mm
\noindent
We also observe that the optimal choice of the dimension parameter 
$\tau$ depends on the distribution of the eigenvalues of the 
covariance operator $\Sigma$. Nevertheless, it is possible to upper 
bound what happens when this distribution of eigenvalues 
is the worst possible. 
Indeed, observe that
\[
 \sum_{i=\tau+1}^{+\infty}\lambda_i^2 \leq \lambda_{\tau+1}\mathbf{Tr}(\Sigma)
 \]
 and also $\tau\lambda_{\tau+1} \leq \mathbf{Tr}(\Sigma),$ so that 
  \[
 \sum_{i=\tau+1}^{+\infty}\lambda_i^2 \leq \tau^{-1}\mathbf{Tr}(\Sigma)^2.
 \]
Hence the worst case formulation of the above proposition is obtained choosing 
\[
 \tau= \Biggl\lceil \frac{ \mathbf{Tr}(\Sigma)}{ \sqrt{2} B(\lambda_1)} 
\Biggr\rceil. 
\]
We obtain 
\begin{cor}
With probability at least $1-\epsilon,$ for any $1/L$-Lipschitz function $f$, 
\[
\| f(\Sigma) - f(\widehat \Sigma) \|_\infty \leq 2 L^{-1} \sqrt{\sqrt 2 \, \mathbf{Tr}(\Sigma) B(\lambda_1)+ \, B(\lambda_1)^2}.
\]
\end{cor}

\vskip2mm
\noindent
The above result shows that, in case we use the bound $B$ presented in Proposition~\ref{prop1.23eig},
the worst case speed is not slower than $n^{-1/4}$.
We do not know whether this rate is optimal in the worst case.

\vskip2mm
\noindent
We conclude the section by observing that it is possible to obtain a similar bound also in the case of the Frobenius norm.
To do this, we consider the estimator 
\[
\widetilde \Sigma u = \sum_{i=1}^{+\infty} \widetilde \lambda_i \ \langle u, q_i\rangle_{\mathcal H} \ q_i
\]
with the same eigenvectors $\{q_i\}_{i\geq 1}$ as $\widehat \Sigma$ and eigenvalues 
\[
\widetilde \lambda_i = \left[\widehat \lambda_i -B(\widehat \lambda_i) \right]_+.
\] 
We observe that, in the event of probability at least $1-\epsilon$ described in equation~\eqref{eq_eigen}, for any $i \geq 1$,
\[
\widetilde \lambda_i \leq \lambda_i.
\]
The following proposition holds. 

\begin{prop}\label{f_norm} With probability at least $1-\epsilon$, for any $1/L$-Lipschitz 
function $f$, 
\begin{align*}
\| f(\Sigma) - f(\widetilde \Sigma) \|_{HS}
& \leq L^{-1} \| \Sigma-\widetilde \Sigma\|_{HS} \\
&  \leq \min_{\tau \geq 1} L^{-1} \sqrt{13 \, \tau B(\lambda_1)^2+ 2 \sum_{i=\tau+1}^{+\infty }\lambda_i^2}
\end{align*}
 where $B$ is defined in equation ~\eqref{eq_eigen} and $\lambda_1$ is the largest eigenvalue of $\Sigma$.
\end{prop}

\vskip1mm
\noindent
For the proof we refer to section~\ref{pf4}.\\[1mm]
Similarly as already done in the case of the operator norm, 
it is possible to provide a worst case reformulation of Proposition~\ref{f_norm}.
In this case the dimension parameter $\tau$ is chosen as
\[
 \tau= \Bigl\lceil \sqrt{2/13} \ \mathbf{Tr}(\Sigma) B(\lambda_1)^{-1} \Bigr\rceil.
\]
We obtain
\begin{cor}
With probability at least $1 - \epsilon$, for any $1/L$-Lipschitz function 
$f$, 
\[ 
\lVert f(\Sigma) - f(\widetilde{\Sigma}) \rVert_{HS} \leq 
L^{-1} \sqrt{ 11 \mathbf{Tr}(\Sigma) B(\lambda_1) + 13 B(\lambda_1)^2}.
\] 
\end{cor}

\newpage
\section{Proofs}

Before proving the results presented in the previous sections, we recall some notation. 
Let $\mathcal S$ be unit sphere of $\mathcal H$.
Let $\lambda_1\geq \lambda_2 \geq \dots$ (resp. $ \widehat \lambda_1\geq \widehat \lambda_2 \geq \dots$) denote the eigenvalues of $\Sigma$ 
(resp. $ \widehat \Sigma$)
and $\{ p_i\}_{i\geq 1}$ (resp. $\{ q_i\}_{i\geq 1}$) a corresponding orthonormal basis of eigenvectors,
so that 
\begin{align*}
\Sigma u & = \sum_{i=1}^{+\infty} \lambda_i \ \langle u, p_i\rangle_{\mathcal H}\ p_i\\
\widehat \Sigma u & = \sum_{i=1}^{+\infty} \widehat \lambda_i \ \langle u, q_i\rangle_{\mathcal H}\ q_i.
\end{align*}

\subsection{Proof of Proposition~\ref{prop1.23eig}}\label{pf_eig}
We recall that the proposition is formulated first 
in the finite-dimensional setting of $\mathbb R^d$ that 
will be used in the proof. 
We observe that once we have proved the proposition for the eigenvalues of $\widehat \Sigma$, 
the result also holds for their positive parts, that are the eigenvalues of $\widehat \Sigma_+$. 
In this section, $\big\{ \widehat\lambda_i\big\}_{i=1}^d$
denote (with a change of notation) the eigenvalues of $\widehat \Sigma$. 

\vskip2mm
\noindent
Observe that, for any $i \in \{ 1,\dots, d\},$ the vector space 
\[
\mathbf{span}\{q_1,\dots, q_{i-1} \}^{\perp} \cap \mathbf{span} 
\{p_1,\dots, p_i\} \subset \mathbb R^d
\]
is of dimension at least 1, so that the set
\[
V_i = \bigl\{ \theta \in \mathbb S_d \ | \ \theta \in 
\mathbf{span} \{q_1,\dots, q_{i-1} \}^{\perp} \cap 
\mathbf{span} \{p_1,\dots, p_i\} \bigr\} \subset \mathbb R^d
\]
is non-empty. 
Indeed, putting $A  = \mathbf{span} \{ q_1, \dots, q_{i-1} \}^{\perp}$ 
and $B = \mathbf{span} \{p_1, \dots, p_i\}$, 
we see that 
\[\dim(A \cap B) = \dim(A) + \dim(B) - \dim(A+B) \geq 1,\] 
since $\dim(A+B) \leq \dim(\mathbb{R}^d) = d$ and $\dim(A) + \dim(B) = d+1$. 
Hence, there exists $\theta_i \in V_i$ 
and, for such $\theta_i$, we have $\theta_i^{\top}\Sigma \theta_i\geq \lambda_i.$ It follows that
\begin{align*}
\max\{ \lambda_i , \sigma\} & \leq \sup \left\{ \max\{\theta^{\top}\Sigma\theta, \sigma\}   \ | \  \theta \in V_i \right\}  \\
& \leq \sup \Bigl\{ \max\{\theta^{\top}\Sigma\theta , \sigma\}  \ | \  \theta \in \mathbb{S}_d,\ \theta \in \mathbf{span} \{q_1, \dots, q_{i-1} \}^{\perp} \Bigr\} .
\end{align*}
Therefore, according to Proposition \ref{prop_giu}, 
\begin{multline*}
\max \{ \lambda_i, \sigma \} \bigl( 1 - 2 \eta(\lambda_i) \bigr) \\ 
\leq \sup \Bigl\{ \max \{ \theta^{\top} \widehat \Sigma \theta, \sigma \} \; | \;  
\theta \in \mathbb{S}_d \cap \mathbf{span} \{ q_1, \dots, q_{i-1} \}^{\perp} \Bigr\}+ 5 \delta \lVert \Sigma \rVert_{HS} \\ 
 \leq \max \{ \widehat{\lambda}_i , \sigma \} + 5 \delta \lVert \Sigma \rVert_{HS}. 
\end{multline*}
In the same way,
\begin{align*}
\max \{ \widehat{\lambda}_i, \sigma \} & \leq \sup \Bigl\{ 
\max \{ \theta^{\top}\Sigma\theta, \sigma \}  \Bigl( 1 + 2 \eta\bigl( \theta^{\top}\Sigma\theta \bigr) \Bigr) \; 
\\ & \hspace{20ex} | \; \theta \in \mathbb{S}_d \cap \mathbf{span} \{ p_1, \dots, p_{i-1}\}^{\perp} \Bigr\} 
+ 5 \delta \lVert \Sigma \rVert_{HS}\\ 
& \leq \max \{ \lambda_i, \sigma \} \bigl( 1  + 2 \eta (\lambda_i) \bigr) 
+ 5 \delta \lVert \Sigma \rVert_{HS}, 
\end{align*}
\begin{multline*} 
\max \{ \widehat{\lambda}_i, \sigma \} \Bigl( 1 - 2 \eta \bigl( 
\min\{\widehat{\lambda}_i, s_4\} \bigr) \Bigr)\\
\leq \sup \Bigl\{ \max \{ \theta^{\top}\Sigma\theta , \sigma \} \; | \; \theta \in \mathbb{S}_d \cap 
\mathbf{span} \{ p_1, \dots, p_{i-1} \}^{\perp} \Bigr\} + 5 \delta \lVert \Sigma \rVert_{HS} \\
\leq \max \{ \lambda_i, \sigma \} + 5 \delta \lVert \Sigma \rVert_{HS},
\end{multline*}
and
\begin{align*}
\max \{ \lambda_i, \sigma \} &
\leq \sup \Bigl\{ \max \{ \theta^{\top} \widehat \Sigma \theta, \sigma \} \Bigl( 1 + 2 \eta \bigl( \min\{ \theta^{\top} \widehat \Sigma \theta, s_4\} \bigr) \Bigr) \; | \\
& \hskip21mm | \; \theta \in \mathbb{S}_d \cap \mathbf{span} \{ q_1, \dots, q_{i-1} \}^{\perp} \Bigr\} 
+ 5 \delta \lVert \Sigma \rVert_{HS} \\ 
& \leq \max \{ \widehat{\lambda}_i, \sigma \} \Bigl( 1 + 2 \eta \bigl(\min\{ \widehat{\lambda}_i, s_4\} \bigr) \Bigr) + 5 \delta \lVert 
\Sigma \rVert_{HS}.
\end{align*}  
In all these inequalities we have used the fact that the functions
\begin{align*}
t & \mapsto \max \{t, \sigma \}  \Bigl( 1 - 2\ \eta\big(\min\{t,s_4\} \big) \Bigr) \\ 
t & \mapsto \max \{t, \sigma \}  \Bigl( 1 + 2\ \eta\big(\min\{t,s_4\} \big) \Bigr) 
\end{align*}
are non-decreasing.\\[1mm]
To prove the second part of the proposition, it is sufficient to observe that 
\[
|\lambda_i - \widehat \lambda_i |  \leq \bigl\lvert \max\{ \lambda_i , 
\sigma\}- \max\{ \widehat{\lambda}_i , \sigma\}\bigr\rvert + \sigma.
\]

\subsection{Some technical results}

In this section we introduce and prove two 
technical results that we use later in the proofs.

\begin{lem}\label{lem1.29} 
Assume that equation~\eqref{b_B} holds true. 
With probability at least $1-\epsilon,$ for any $k\geq1,$ the two following 
inequalities hold together
\begin{align}
\label{lem1.29eq1}
 \sum_{i=1}^{+\infty} \left( \lambda_i-\lambda_k \right)^2 \langle q_k, p_i \rangle^2 & \leq 2B\left(\lambda_1 \right)^2 \\
  \label{lem1.29eq2}
 \sum_{i=1}^{+\infty} \left( \lambda_i-\widehat \lambda_k \right)^2 \langle q_k, p_i \rangle^2 & \leq B\left(\lambda_1 \right)^2.
 \end{align}
\end{lem}

\begin{proof} We observe that, according to equation~\eqref{b_B}, with probability at least $1-\epsilon,$
\[
\sup_{u \in \mathcal{S}} \| \Sigma u - \widehat \Sigma u \|_{\mathcal H}  \leq  B\left(\lambda_1 \right).
\]
To prove equation~\eqref{lem1.29eq2} it is sufficient to observe that, since 
\[
\bigl\| \Sigma u - \widehat \Sigma u \bigr\|_{\mathcal H} = \biggl\| \sum_{i,j=1}^{+\infty}(\lambda_i-\widehat \lambda_j) \langle u, q_j \rangle_{\mathcal H}\langle p_i, q_j \rangle_{\mathcal H}\ p_i \biggr\|_{\mathcal H}
\] 
choosing $u = q_k,$ with $k\geq 1$,
 \[
\bigl\| \Sigma q_k - \widehat \Sigma q_k \bigr\|_{\mathcal H}^2 = \sum_{i=1}^{+\infty}(\lambda_i-\widehat \lambda_k)^2 \langle  q_k, p_i \rangle_{\mathcal H}^2 .
\] 
On the other hand, to prove equation~\eqref{lem1.29eq1}, we observe that 
$$\aligned
\bigl\| \Sigma u - \widehat \Sigma u \bigr\|_{\mathcal H} 
& = \biggl\| \sum_{i=1}^{+\infty}  \lambda_i \langle u, p_i\rangle_{\mathcal H}\ p_i -  \sum_{i=1}^{+\infty} \widehat   \lambda_i \langle u, q_i\rangle_{\mathcal H}\ q_i \biggr\|_{\mathcal H}\\
& = \biggl\| \sum_{i=1}^{+\infty}  \lambda_i \left(\langle u, p_i\rangle_{\mathcal H}\ p_i -\langle u, q_i \rangle_{\mathcal H} q_i \right)-  \sum_{i=1}^{+\infty} \left( \widehat   \lambda_i -   \lambda_i \right)\langle u, q_i \rangle_{\mathcal H}\ q_i 
\biggr\|_{\mathcal H}\\
& \geq  \biggl\| \sum_{i=1}^{+\infty}   \lambda_i\left(\langle u, p_i\rangle_{\mathcal H}\ p_i -\langle u, q_i \rangle_{\mathcal H} q_i \right) \biggr\|_{\mathcal H}
- \biggl\| \sum_{i=1}^{+\infty} \left( \widehat   \lambda_i -   \lambda_i \right)\langle u, q_i\rangle_{\mathcal H}\ q_i \biggr\|_{\mathcal H}
\endaligned$$
where, by equation~\eqref{eq_eigen},
\begin{align*} 
\biggl\| \sum_{i=1}^{+\infty} \left( \widehat   \lambda_i -  \lambda_i \right)\langle u, q_i\rangle_{\mathcal H}\ q_i \biggr\|_{\mathcal H}^2 
& = \sum_{i=1}^{+\infty} \left( \widehat   \lambda_i -   \lambda_i \right)^2\langle u, q_i\rangle_{\mathcal H}^2 \leq B\left(\lambda_1 \right)^2.
\end{align*}
Choosing again $u=  q_k,$ for $k \geq1,$ we get
\begin{align*}
\biggl\| \sum_{i=1}^{+\infty}  \lambda_i \Bigl( \langle q_k, 
p_i\rangle_{\mathcal H}\ p_i -\langle q_k, q_i\rangle_{\mathcal H}\ q_i 
\Bigr) \biggr\|_{\mathcal H}^2 
& = \biggl\| \sum_{i,j=1}^{+\infty} \left( \lambda_i-\lambda_j \right) \langle q_k, q_j\rangle_{\mathcal H} \langle q_j, p_i \rangle_{\mathcal H}\  p_i
\biggr\|_{\mathcal H}^2\\
& =  \biggl\| \sum_{i=1}^{+\infty} \left( \lambda_i-\lambda_k \right) \langle q_k, p_i \rangle_{\mathcal H}\ p_i \biggr\|_{\mathcal H}^2\\
& =  \sum_{i=1}^{+\infty} \left( \lambda_i-\lambda_k \right)^2 \langle q_k, p_i \rangle_{\mathcal H}^2,
\end{align*}
which concludes the proof.
\end{proof}

\begin{prop}\label{fbn} 
Let $M,$ $M'$ be two Hilbert-Schmidt operators so that
\begin{align*}
M u & = \sum_{i=1}^{+\infty} \mu_i \ \langle u, p_i\rangle_{\mathcal H} \ p_i \\
M' u & = \sum_{i=1}^{+\infty}\mu'_i \ \langle u, q_i \rangle_{\mathcal H} \ q_i
\end{align*}
where $\{\mu_i\}_{i\geq1}$ are the eigenvalues of $M$ with respect to the orthonormal basis of eigenvectors $p_1,\dots, p_d$ 
and $\{\mu_i'\}_{i\geq1}$ the eigenvalues of $M'$ with respect to the orthonormal basis of eigenvectors $q_1,\dots, q_d.$ 
We have 
\begin{equation}\label{eq_fbn} \| M-M'\|_{HS}^2 = \sum_{i,k=1}^{+\infty} (  \mu_i -  \mu'_k)^2 \langle p_i, q_k\rangle_{\mathcal H}^2.
\end{equation}
Moreover, given $f$ a ${1}/{L}$-Lipschitz function, 
\begin{equation}\label{rcorf} \| f(M)-f(M')\|_{HS} \leq \frac{1}{L} \| M-M'\|_{HS} .
\end{equation}
\end{prop}

\begin{proof}
By definition  
\begin{align*}
(M- M' )u & =  \sum_{i=1}^{+\infty} \mu_i\  \langle u, p_i \rangle_{\mathcal H} \ p_i - \sum_{k=1}^{+\infty} \mu'_k \  \langle u,q_k \rangle_{\mathcal H} \ q_k\\
& = \sum_{i,k=1}^{+\infty} (  \mu_i -  \mu'_k) \langle p_i, q_k\rangle_{\mathcal H} \langle u, q_k\rangle_{\mathcal H} \ p_i
\end{align*}
and its Hilbert-Schmidt norm is
\begin{align*}
\| M-M'\|_{HS}^2 & = \mathbf{Tr}((M-M')^* (M-M'))\\
& = \sum_{n= 1}^{+\infty} \big\langle (M-M')q_n,  (M-M')q_n \big\rangle_{\mathcal H}\\
& = \sum_{i,k=1}^{+\infty} (  \mu_i -  \mu'_k)^2 \langle p_i, q_k\rangle_{\mathcal H}^2.
\end{align*}

\vskip2mm
\noindent
To prove the second part of the result, it is sufficient to use twice equation~\eqref{eq_fbn}. Indeed 
\begin{align*}
\| f(M)-f(M')\|_{HS}^2 & = \sum_{i,k=1}^{+\infty} ( f( \mu_i) - f( \mu'_k))^2 \langle p_i, q_k\rangle_{\mathcal H}^2\\
& \leq  \frac{1}{L^2}  \sum_{i,k=1}^{+\infty} (  \mu_i -  \mu'_k)^2 \langle p_i, q_k\rangle_{\mathcal H}^2 
= \frac{1}{L^2}\| M-M'\|_{HS}^2. 
\end{align*}
\end{proof}

\subsection{Proof of Proposition~\ref{st_pca}}\label{pf1}

Since $\Pi_r$ and $\widehat \Pi_r$ have the same finite rank, we can write
\[
\|\Pi_r-\widehat \Pi_r\|_{\infty}=\sup_{\substack{u \in \mathcal S \\ u \in \mathbf{Im}(\widehat \Pi_r)}} \big\| \Pi_ru-\widehat \Pi_ru\big\|_{\mathcal H}
\]
as shown in Lemma ~\ref{imQ} in Appendix ~\ref{appx} (the proof is made in 
the finite dimensional setting, but easily extends to the infinite dimensional 
setting, since the orthogonal of $\mathbf{Im} (P) \oplus \mathbf{Im} (Q)$ is 
included in $\mathbf{ker}(P) \cap \mathbf{ker}(Q)$).\\
Moreover, for any $u\in \mathbf{Im}(\widehat \Pi_r)\cap \mathcal{S},$  we observe that
$$\aligned \big\| \Pi_ru-\widehat \Pi_ru\big\|_{\mathcal H}^2 & = \big\| \Pi_ru-u\big\|_{\mathcal H}^2 \\
& = \biggl\| \sum_{i=1}^r \langle u, p_i\rangle_{\mathcal H}\ p_i - \sum_{i=1}^{+\infty} \langle u, p_i\rangle_{\mathcal H}\ p_i \biggr\|_{\mathcal H}^2\\
& = \sum_{i=r+1}^{+\infty} \langle u, p_i\rangle_{\mathcal H}^2.
\endaligned$$

\noindent
Since any $u \in \mathbf{Im}(\widehat \Pi_r)$ can be written as $u= \sum_{k=1}^r \langle u, q_k\rangle_{\mathcal H}\ q_k$ with $\sum_{k=1}^r \langle u, q_k\rangle_{\mathcal H}^2=1,$ then
$$\aligned  \big\| \Pi_ru-\widehat \Pi_ru\big\|_{\mathcal H}^2 &= \sum_{i=r+1}^{+\infty} \left( \sum_{k=1}^r \langle u, q_k \rangle_{\mathcal H} \langle q_k,p_i\rangle_{\mathcal H} \right)^2.
\endaligned$$

\noindent
Hence, by the Cauchy-Schwarz inequality, we get
\begin{align} 
\big\| \Pi_ru-\widehat \Pi_ru\big\|_{\mathcal H}^2& \leq  \sum_{i=r+1}^{+\infty} \left( \sum_{k=1}^r \langle u, q_k \rangle_{\mathcal H}^2\right)  \left( \sum_{k=1}^r \langle q_k,p_i\rangle_{\mathcal H}^2\right)  \\
 \label{PtQt}
& = \sum_{k=1}^r \sum_{i=r+1}^{+\infty} \langle q_k,p_i\rangle_{\mathcal H}^2.
\end{align}

\vskip2mm
\noindent
Moreover, for any $k \in  \{1,\dots, r\},$ we have
$$\aligned  
\sum_{i=r+1}^{+\infty} \left(  \lambda_{k}-\lambda_{i} \right)^2 \langle q_k, p_i \rangle_{\mathcal H}^2 & \geq  \sum_{i=r+1}^{+\infty} \left( \lambda_{r}-\lambda_{i} \right)^2 \langle q_k, p_i \rangle_{\mathcal H}^2\\
& \geq \left( \lambda_{r}-\lambda_{r+1} \right)^2 \sum_{i=r+1}^{+\infty}  \langle q_k, p_i \rangle_{\mathcal H}^2 .
\endaligned$$

\noindent
Then, by Lemma~\ref{lem1.29}, with probability at least $1-\epsilon,$
$$\left(\lambda_{r}-\lambda_{r+1} \right)^2 \sum_{i=r+1}^{+\infty}  \langle q_k, p_i \rangle_{\mathcal H}^2 \leq 2B\left(\lambda_1 \right)^2.$$

\noindent
Applying the above inequality to equation ~\eqref{PtQt}
we conclude the proof.

\subsection{Proof of Proposition~\ref{supnorm}}\label{pf3}

Assume that the event of probability at least $1-\epsilon$ described in equation~\eqref{eq_eigen} holds true. 
Define the operator
\[
H u= \sum_{k=1}^{+\infty} \lambda_k \; \langle u, q_k \rangle_{\mathcal H}\ q_k.
\]
We observe that 
\[
\| f(\Sigma)- f(\widehat \Sigma) \|_{\infty} \leq \| f(\Sigma)- f(H) \|_{\infty} + \| f(H)- f(\widehat \Sigma) \|_{\infty}
\]

\noindent
and we look separately at the two terms. By definition of the operator norm, 
we have 
$$ \aligned\| f(H)- f(\widehat \Sigma) \|_{\infty}^2 & = \sup_{u\in \mathcal S} \big\| f(H)u- f(\widehat \Sigma)u\big \|_{\mathcal H}^2\\
& =  \sup_{u\in \mathcal S} \Big\| \sum_{k=1}^{+\infty} (f(\lambda_k) -f(\widehat \lambda_k)) \langle u, q_k \rangle_{\mathcal H}\ q_k \Big\|_{\mathcal H}^2\\
& =  \sup_{u\in \mathcal S} \sum_{k=1}^{+\infty} (f(\lambda_k) -f(\widehat \lambda_k))^2 \langle u, q_k \rangle_{\mathcal H}^2.
\endaligned$$ 

\noindent
Since the function $f$ is $1/L$-Lipschitz, we get
$$ \| f(H)- f(\widehat \Sigma) \|_{\infty}^2  \leq L^{-2}  \sup_{u\in \mathcal S} \sum_{k=1}^{+\infty} (\lambda_k -\widehat \lambda_k)^2 \langle u, q_k \rangle_{\mathcal H}^2$$
\noindent
and then, applying equation ~\eqref{eq_eigen},  
we obtain 
\[
\| f(H)- f(\widehat \Sigma) \|_{\infty}^2  \leq L^{-2} B\left(\lambda_1 \right)^2.
\] 

\noindent
On the other hand, we have 
$$\aligned  \| f(\Sigma)- f(H) \|_{\infty} &  \leq  \| f(\Sigma)- f(H) \|_{HS}\\
& \leq \frac{1}{L} \|\Sigma-H\|_{HS},
\endaligned$$
as shown in equation ~\eqref{rcorf}. Hence, according to Proposition~\ref{fbn}, we get
\[
\| f(\Sigma)- f(H) \|_{\infty}^2 \leq  \frac{1}{L^2} \sum_{i,k=1}^{+\infty} (\lambda_i-\lambda_k)^2 \langle p_i, q_k\rangle_{\mathcal H}^2
\]
where, for any $\tau \geq 1,$
\begin{align*}&\sum_{i,k=1}^{+\infty} (\lambda_i-\lambda_k)^2 \langle p_i, q_k\rangle_{\mathcal H}^2 \leq \left( \sum_{i=1}^\tau \sum_{k=1}^{+\infty}+ \sum_{i=1}^{+\infty} \sum_{k=1}^\tau+ \sum_{i,k=\tau+1}^{+\infty}\right) (\lambda_i-\lambda_k)^2 \langle p_i, q_k\rangle_{\mathcal H}^2.
\end{align*}

\noindent
Since $\lambda_i \geq 0,$ for any $i\geq1,$ we get
$$\aligned \sum_{i,k=\tau+1}^{+\infty} (\lambda_i-\lambda_k)^2 \langle p_i, q_k\rangle_{\mathcal H}^2 
& \leq 2\sum_{i=\tau+1}^{+\infty} \lambda_i^2.
\endaligned$$

\noindent
Moreover, by Lemma~\ref{lem1.29}, we have 
\[
 \sum_{k=1}^\tau\sum_{i=1}^{+\infty}  (\lambda_i-\lambda_k)^2 \langle p_i, q_k\rangle_{\mathcal H}^2
 \leq 2 \tau  B\left(\lambda_1 \right)^2
\]

\noindent
and since the same bound also holds for
\[
\sum_{i=1}^\tau \sum_{k=1}^{+\infty} (\lambda_i-\lambda_k)^2 \langle p_i, q_k\rangle_{\mathcal H}^2, 
\]
we conclude the proof.

\subsection{Proof of Proposition~\ref{f_norm}}\label{pf4}

The first inequality follows from equation ~\eqref{rcorf}.
Thus it is sufficient to prove that 
\[
\| \Sigma-\widetilde \Sigma\|_{HS}
\leq \sqrt{13 \, \tau B(\lambda_1)^2+ 2 \sum_{i=\tau+1}^{+\infty }\lambda_i^2}.
\]
According to  Proposition~\ref{fbn}, we have 
\[ 
\| \Sigma - \widetilde \Sigma\|_{HS}^2 = \sum_{i,k=1}^{+\infty} (  \lambda_i -  \widetilde \lambda_k)^2 \langle p_i, q_k\rangle_{\mathcal H}^2,
\]
where
\[
\sum_{i,k=1}^{+\infty} (\lambda_i - \widetilde \lambda_k)^2 \langle p_i, q_k\rangle_{\mathcal H}^2 
\leq  \left(\sum_{i=1}^{\tau}\sum_{k=1}^{+\infty}+ \sum_{k=1}^{\tau}\sum_{i=1}^{+\infty}+ \sum_{i,k=\tau+1}^{+\infty}  \right)(\lambda_i - \widetilde \lambda_k)^2 \langle p_i, q_k\rangle_{\mathcal H}^2.
\]

\noindent
Since, by definition, $\widetilde \lambda_i \leq  \lambda_i,$ it follows that
\begin{align*}
\sum_{i,k=\tau+1}^{+\infty} (\lambda_i - \widetilde \lambda_k)^2 \langle p_i, q_k\rangle_{\mathcal H}^2 & \leq \sum_{i,k=\tau+1}^{+\infty} (\lambda_i^2 + \widetilde \lambda_k^2) \langle p_i, q_k\rangle_{\mathcal H}^2 \\
& \leq 2 \sum_{i=\tau+1}^{+\infty} \lambda_i^2.
\end{align*}
Furthermore, we observe that 

\[
\sum_{k=1}^{\tau}\sum_{i=1}^{+\infty} (\lambda_i - \widetilde \lambda_k)^2 \langle p_i, q_k\rangle_{\mathcal H}^2 
\leq 2 \sum_{k=1}^{\tau}\sum_{i=1}^{+\infty} (\lambda_i - \widehat \lambda_k)^2 \langle p_i, q_k\rangle_{\mathcal H}^2  
+2 \sum_{k=1}^{\tau}\sum_{i=1}^{+\infty}  B(\widehat \lambda_k)^2  \langle p_i, q_k\rangle_{\mathcal H}^2 ,
\]

\noindent
where, by Lemma~\ref{lem1.29}, 
\[
\sum_{k=1}^{\tau}\sum_{i=1}^{+\infty} (\lambda_i - \widehat \lambda_k)^2 \langle p_i, q_k\rangle_{\mathcal H}^2 \leq \tau B(\lambda_1)^2.
\]
and $B(\widehat \lambda_k)  \leq B(\widehat \lambda_1).$
We have then proved that 
\[
\sum_{k=1}^{\tau}\sum_{i=1}^{+\infty} (\lambda_i - \widetilde \lambda_k)^2 \langle p_i, q_k\rangle_{\mathcal H}^2 
\leq 2 \tau B(\lambda_1)^2+ 2\tau B(\widehat \lambda_1)^2.
\]
Using the fact that $B(t + a) \leq B(t) + a/2$,
we deduce that 
\[
B(\widehat \lambda_1) \leq B\bigl[\lambda_1 + B(\lambda_1) \bigr] 
\leq 3 B(\lambda_1)/2.
\]
This proves that 
\[
\sum_{k=1}^{\tau}\sum_{i=1}^{+\infty} (\lambda_i - \widetilde \lambda_k)^2 \langle p_i, q_k\rangle_{\mathcal H}^2  
\leq 2 \tau \left[ B(\lambda_1)^2+ 9B( \lambda_1)^2/4 \right] = 13 \tau B( \lambda_1)^2/2.
\]
Considering that the same bound holds for
\[
\sum_{i=1}^{\tau}\sum_{k=1}^{+\infty}(\lambda_i - \widetilde \lambda_k)^2 \langle p_i, q_k\rangle_{\mathcal H}^2 ,
\]
we conclude the proof.

\newpage
\appendix

\section{Orthogonal Projectors}\label{appx}
In this appendix we introduce some results on orthogonal projectors. \\[1mm]
Let $P,\ Q: \mathbb R^d \to \mathbb R^d$ be two orthogonal projectors. 
We denote by $\mathbb S_d$ the unit sphere of $\mathbb R^d.$
By definition, 
$$\|P-Q\|_{\infty}=\sup_{x \in \mathbb S_d } \|Px-Qx\|$$
where, without loss of generality, we can take the supremum over the normalized eigenvectors of $P-Q.$ 

\vskip 2mm
\noindent
A good way to describe the geometry of $P-Q$ is to consider the 
eigenvectors of $P+Q$. 
\begin{lem}\label{lemma1} Let $x\in \mathbb S_d$ be an eigenvector of $P+Q$ with eigenvalue $\lambda.$ 
\begin{enumerate}
\item \label{eigen:1} In the case when $\lambda = 0$, then $Px = Qx = 0$, 
so that $x \in \mathbf{ker}(P) \cap \mathbf{ker}(Q)$; 
\item \label{eigen:2} in the case when $\lambda = 1$, then $PQx = QPx = 0$, 
so that 
\[
x \in \mathbf{ker}(P) \cap \mathbf{Im}(Q)  \oplus \mathbf{Im}(P) \cap \mathbf{ker}(Q); 
\]
\item \label{eigen:3} in the case when $\lambda = 2$, then $x = Px = Qx$, 
so that $x \in \mathbf{Im}(P) \cap \mathbf{Im}(Q)$; 
\item \label{eigen:4} otherwise, when $\lambda \in ]0,1[ \cup ]1,2[$, 
\[ 
(P-Q)^2 x = (2 - \lambda)\lambda x \neq 0, 
\] 
so that $(P-Q)x \neq 0$. Moreover 
\[ (P+Q)(P-Q) x = (2-\lambda)(P-Q)x,\]
so that $(P-Q) x$ is an eigenvector of $P+Q$ with eigenvalue $2-\lambda$.
Moreover
\[
0 < \lVert Px \rVert = \lVert Qx \rVert < \lVert 
x \rVert,
\] 
$x - Px \neq 0$, and $\bigl( Px, x - Px \bigr)$ is an orthogonal 
basis of $\mathbf{span} \bigl\{ x, (P-Q)x \bigr\}$.
\end{enumerate}
\end{lem}
\begin{proof}
The operator $P+Q$ is symmetric, positive semi-definite, and $\lVert P + Q 
\rVert \leq 2$, so that there is a basis of eigenvectors and 
all eigenvalues are in the intervall $[0,2]$. \\[1mm]
In case \ref{eigen:1}, 
$ 0 = \langle Px + Qx, x \rangle = \lVert Px \rVert^2 + \lVert Q x \rVert^2$, 
so that $Px = Qx = 0$. \\[1mm]
In case \ref{eigen:2}, $PQx = P(x - Px) = 0$ and similarly 
$QPx = Q(x - Qx) = 0$. \\[1mm]
In case \ref{eigen:3}, 
\[
\lVert Px \rVert^2 + \lVert Qx \rVert^2 = \langle (P+Q)x, x \rangle = 
2 \langle x , x \rangle = \lVert Px \rVert^2 + \lVert x - Px \rVert^2  
+ \lVert Qx \rVert^2 + \lVert x - Qx \rVert^2 , 
\]
so that $\lVert x - Px \rVert = \lVert x - Qx \rVert = 0$. \\[1mm]
In case \ref{eigen:4}, remark that 
\[
PQx = P(\lambda x - Px) = (\lambda - 1) Px
\] 
and similarly $QPx = Q ( \lambda x - Q x) = (\lambda - 1)Q x$. 
Consequently 
\[
(P-Q)(P-Q) x = (P - QP - PQ + Q) x = (2 - \lambda) (P+Q) x = (2 - \lambda)\lambda x \neq 0, 
\]
so that $(P-Q)x \neq 0$. Moreover
\[
(P+Q)(P-Q)x = (P - PQ + QP - Q)x = (2 - \lambda)(P-Q)x. 
\]
Therefore $(P-Q)x$ is an eigenvector of $P+Q$ with eigenvalue $2 - \lambda 
\neq \lambda$, so that $\langle x , (P - Q) x \rangle = 0$, since $P + Q$ 
is symmetric. As $\langle x , (P - Q) x \rangle = \lVert Px \rVert^2 
- \lVert Qx \rVert^2$, this proves that $\lVert Px \rVert = \lVert 
Qx \rVert$. Since $(P+Q)x = \lambda x \neq 0$, necessarily $\lVert Px \rVert 
= \lVert Qx \rVert > 0$. Observe now that 
\[ 
\lVert Px \rVert^2 = \frac{1}{2} \bigl( 
\lVert Px \rVert^2 + \lVert Q x \rVert^2 \bigr) 
= \frac{1}{2} \langle x , (P+Q)x \rangle = 
\frac{\lambda}{2} \lVert x \rVert^2 < \lVert x \rVert^2. 
\]     
Therefore $\lVert x - Px \rVert^2 = \lVert x \rVert^2 - \lVert P x \rVert^2
> 0$, proving that $x - Px \neq 0$. Similarly, since $P$ and $Q$ play 
symmetric roles, $\lVert Q x \rVert < \lVert x \rVert$ and 
$x - Qx \neq 0$. \\[1mm]
As $P$ is an orthogonal projector, $(Px, x-Px)$ is an orthogonal pair 
of non-zero vectors. Moreover 
\[
x = x - Px + Px \in \mathbf{span} \{ Px, x - Px \}
\]
and 
\[
(P-Q)x = 2 Px - \lambda x = (2 - \lambda) Px - \lambda(x - Px) 
\in \mathbf{span} \{ Px, x - Px \}\]
therefore, $( Px, x - Px )$ is an 
orthogonal basis of $\mathbf{span}\{x, (P-Q)x \}$.
\end{proof}
\begin{lem}
\label{lemma2}
There is an orthonormal basis $(x_i)_{i= 1}^d$ of eigenvectors 
of $P+Q$ with corresponding eigenvalues 
$\{\lambda_i, \ i=1, \dots d \}$ and indices $2 m \leq  p \leq  q \leq s$, such that 
\begin{enumerate}
\item $\lambda_{i} \in ]1,2[$, if $1 \leq i \leq m$,
\item $\lambda_{m+i} = 2 - \lambda_i$, if $1 \leq i \leq m$, 
and $x_{m+i} = \lVert (P-Q) x_{i} 
\rVert^{-1} (P-Q) x_{i}$, 
\item $x_{2m+1}, \dots, x_{ p} \in \bigl(\mathbf{Im} (P) \cap \mathbf{ker} (Q)\bigr)$, 
and $\lambda_{2m+1} = \dots = \lambda_{p} = 1$, 
\item $x_{p+1}, \dots, x_{q} \in \bigl( \mathbf{Im} (Q) \cap \mathbf{ker} (P) 
\bigr)$, and $\lambda_{p + 1} = \dots =  \lambda_{q} = 1$, 
\item $x_{q + 1}, \dots, x_{s} \in \mathbf{Im}(P) \cap \mathbf{Im}(Q)$, 
and $\lambda_{q+1} = \dots =  \lambda_{s} = 2$, 
\item $x_{s+1}, \dots, x_d \in \mathbf{ker}(P) \cap \mathbf{ker}(Q)$, 
and $\lambda_{s+1} = \dots = \lambda_d = 0$.
\end{enumerate}
\end{lem}
\begin{proof}
As already explained at the beginning of proof of Lemma ~\ref{lemma1}, there exists a basis of eigenvectors of $P+Q$.
From the previous lemma, 
we have that all eigenvectors in the kernel of $P+Q$ are in 
$\mathbf{ker}(P) \cap \mathbf{ker}(Q)$, and on the other hand obviously 
$\mathbf{ker}(P) \cap \mathbf{ker}(Q) \subset \mathbf{ker}(P+Q)$ so that 
\[
\mathbf{ker}(P+Q) = \mathbf{ker}(P) \cap \mathbf{ker}(Q).
\] 
In the same way 
the previous lemma proves that the eigenspace corresponding 
to the eigenvalue $2$ is equal to $\mathbf{Im}(P) \cap \mathbf{Im}(Q)$. 
It also proves that the eigenspace corresponding to the 
eigenvalue $1$ is included in and consequently is equal 
to $\bigl( \mathbf{Im}(P) \cap \mathbf{ker}(Q) \bigr)  \oplus 
\bigl( \mathbf{ker}(P) \cap \mathbf{Im}(Q) \bigr)$, so that we can form an 
orthonormal basis of this eigenspace by taking the union 
of an orthonormal basis of $\mathbf{Im}(P) \cap \mathbf{ker}(Q)$ and 
an orthonormal basis of $\mathbf{ker}(P) \cap \mathbf{Im}(Q)$. \\[1mm]
Consider now an eigenspace corresponding to an eigenvalue 
$\lambda \in ]0,1[ \cup ]1,2[$ and let $x, y$ be two orthonormal eigenvectors  
in this eigenspace. 
From the previous lemma, remark that 
\[ 
\langle (P-Q) x, (P-Q) y \rangle = \langle (P-Q)^2 x, y \rangle = 
(2 - \lambda) \lambda \langle x, y \rangle = 0.
\]  
Therefore, if $x_1, \dots, x_k$ is an orthonormal basis of the eigenspace $V_{\lambda}$ corresponding to the eigenvalue $\lambda$, 
then $(P-Q)x_1, \dots, (P-Q)x_k$ is an orthogonal system in $V_{2 - \lambda}$. 
If this system was not spanning $V_{2 - \lambda}$, we could add to 
it an orthogonal unit vector $y_{k+1} \in V_{2 - \lambda}$
so that $x_1, \dots, x_k, (P-Q) y_{k+1}$ would be an orthogonal 
set of non-zero vectors in $V_{\lambda}$, which would contradict 
the fact that $x_1, \dots, x_k$ was supposed to be an orthonormal basis of $V_{\lambda}$. 
Therefore,  
\[ \Bigl( \lVert (P-Q) x_i \rVert^{-1} (P-Q) x_i, \, 1 \leq i \leq k \Bigr) 
\]
is an orthonormal basis of $V_{2 - \lambda}$. Doing this construction 
for all the eigenspaces $V_{\lambda}$ such that $\lambda \in ]0,1[$ 
achieves the construction of the orthonormal basis described in 
the lemma.  
\end{proof}
\begin{lem}
Consider the orthonormal basis of the previous lemma. 
The set of vectors 
\[ 
\bigl( Px_1, \dots, Px_m, x_{2m+1}, \dots, x_p, x_{q+1}, \dots, x_s \bigr)
\] 
is an orthogonal basis of $\mathbf{Im}(P)$. 
The set of vectors 
\[ 
\bigl( Qx_1, \dots, Qx_m, x_{p+1}, \dots, x_q, x_{q+1}, \dots, x_s \bigr)
\] 
is an orthogonal basis of $\mathbf{Im}(Q)$. 
\end{lem}
\begin{proof}
According to Lemma \ref{lemma1}, $(Px_i, x_i - Px_i)$ is an orthogonal basis 
of $\mathbf{span} \{ x_i, x_{m+i} \}$, 
so that 
\[
\bigl( Px_1, \dots, Px_m, x_1 - Px_1, \dots, x_m - Px_m, x_{2m+1}, \dots, x_d
\bigr)
\]
is another orthogonal basis of $\mathbb{R}^d$. 
Each vector of this basis is either in $\mathbf{Im}(P)$ or in $\mathbf{ker}(P)$ and
more precisely
\begin{align*}
Px_1, \dots, Px_m, x_{2m+1}, \dots, x_p, x_{q+1}, \dots, x_s & \in \mathbf{Im}(P),\\ 
x_1 - Px_1, \dots, x_m - Px_m, x_{p+1}, \dots, x_q, 
x_{s+1}, \dots, x_d & \in \mathbf{ker}(P).
\end{align*}
This proves the claim of the lemma concerning $P$. Since $P$ and $Q$ 
play symmetric roles, this proves also the claim concerning $Q$, 
{\em mutatis mutandis}. 
\end{proof}
\begin{cor}
\label{lemma4}
The projectors $P$ and $Q$ have the same rank if and only if 
\[p - 2m = q - p.\] 
\end{cor} 

\begin{lem}
\label{imQ}
Assume that $\mathrm{rk}(P) = \mathrm{rk}(Q)$. Then 
\[ 
\lVert P - Q \rVert_{\infty} = \sup_{\theta \in \mathbf{Im}(Q) \cap \mathbb{S}_d}  
\lVert (P-Q) \theta \rVert.
\] 
\end{lem}
\begin{proof}
As $P-Q$ is a symmetric operator, we have
\begin{multline*}
\sup_{\theta \in \mathbb{S}_d} \lVert (P-Q) \theta \rVert^2 = 
\sup \Bigl\{ \langle (P-Q)^2 \theta, \theta \rangle \ | \ \theta \in \mathbb{S}_d 
\Bigr\} \\ = \sup \Bigl\{ \langle (P-Q)^2 \theta, \theta \rangle \ |  \
\theta \in \mathbb{S}_d \text{ is an eigenvector of } (P-Q)^2 \Bigr\}. 
\end{multline*}
Remark that the basis described in Lemma \ref{lemma2} is also 
a basis of eigenvectors of $(P-Q)^2$. More precisely, 
according to Lemma \ref{lemma1}
\begin{align*}
(P-Q)^2 x_i & = \lambda_i (2 - \lambda_i) x_i, & 1 \leq i \leq m,\\ 
(P-Q)^2 x_{m+i} & = \lambda_i ( 2 - \lambda_i) x_{m+i}, & 1 \leq i \leq m,\\
(P-Q)^2 x_i & = x_i, & 2m < i \leq q, \\ 
(P-Q)^2 x_i & = 0, & q < i \leq d.
\end{align*}
If $q - 2m > 0$, then $\lVert P - Q \rVert_{\infty} = 1$, and  $q-p > 0$,
according to Lemma \ref{lemma4}, so that 
$\lVert (P-Q)x_{p+1} \rVert = 1$, where $x_{p+1} \in \mathbf{Im}(Q)$.  
If $q = 2m$ and $m > 0$, there is $i \in \{1, \dots,  m\}$ 
such that $\lVert P - Q \rVert_{\infty}^2  = \lambda_i(2 - \lambda_i)$. 
Since $x_i$ and $x_{m+i}$ are two eigenvectors of $(P-Q)^2$ corresponding 
to this eigenvalue, all the non-zero vectors in $\mathbf{span} \{x_{i}, x_{m+i} 
\}$ (including $Qx_i$) are also eigenvectors of the same eigenspace. 
Consequently $(P-Q)^2Qx_i = \lambda_i (2 - \lambda_i) Qx_i$, 
proving that 
\[ 
\Bigl\lVert (P-Q) \frac{Q x_i}{\lVert Q x_i \rVert} \Bigr\rVert^2 
= \lambda_i (2 - \lambda_i), 
\] 
and therefore that $\sup_{\theta \in \mathbb{S}_d} \lVert (P-Q) \theta \rVert$ 
is reached in $\mathbf{Im}(Q)$. Finally, if $q = 0$, then $P-Q$ is the null 
operator, so that $\sup_{\theta \in \mathbb{S}_d} \lVert (P-Q) \theta \rVert$
is reached everywhere, including in $\mathbf{Im} (Q) \cap \mathbb{S}_d$. 
\end{proof}

\end{document}